\documentclass[10pt]{article}

\usepackage{amsmath,amssymb,amsfonts,amsthm}
\usepackage{palatino}
\usepackage[euler-digits]{eulervm}
\usepackage{microtype}
\usepackage{tikz,tikz-cd}
\usepackage{graphicx}
\usepackage{fullpage}
\usepackage{ytableau}

\usepackage[margin=1.25in]{geometry}

\newcommand\NN{\mathbb{N}}
\newcommand\QQ{\mathbb{Q}}

\newcommand\Acal{\mathcal{A}}

\newcommand\Hcal{\mathcal{H}}
\newcommand\Lcal{\mathcal{L}}
\newcommand\Wcal{\mathcal{W}}

\newcommand\ii{\mathfrak{I}}
\newcommand\Sym{\mathfrak{S}}

\newcommand\Htil{\widetilde{H}}
\newcommand\Ktil{\widetilde{K}}

\usepackage{amsthm}

\theoremstyle{definition}
\newtheorem{definition}{Definition}[section]
\newtheorem{example}{Example}[section]

\newtheorem{theorem}{Theorem}[section]
\newtheorem{lem}{Lemma}[section]

\newtheorem{cor}{Corollary}[section]
\newtheorem{conj}{Conjecture}[section]
\newtheorem{remark}{Remark}[section]

\definecolor{darkred}{rgb}{0.7,0,0} 
\definecolor{darkgreen}{rgb}{0, .6, 0} 

\title{A proof of the $\frac{n!}{2}$ conjecture for hook shapes}
\author{Sam Armon}
\date{}

\begin{document}

\maketitle

\begin{abstract}
A well-known representation-theoretic model for the transformed Macdonald polynomial $\Htil_\mu(Z;t,q)$, where $\mu$ is an integer partition, is given by the Garsia-Haiman module $\Hcal_\mu$. We study the $\frac{n!}{k}$ conjecture of Bergeron and Garsia, which concerns the behavior of certain $k$-tuples of Garsia-Haiman modules under intersection. In the special case that $\mu$ has hook shape, we use a basis for $\Hcal_\mu$ due to Adin, Remmel, and Roichman to resolve the $\frac{n!}{2}$ conjecture by constructing an explicit basis for the intersection of two Garsia-Haiman modules.
\end{abstract}

\section{Introduction}

The Macdonald polynomials $\Htil_\mu(Z;t,q)$ have become a central object of study in algebra, combinatorics, and geometry since their introduction by Macdonald in \cite{Mac88}. Of particular import is the expansion of $\Htil_\mu(Z;t,q)$ in the Schur basis:
\[
\Htil_\mu(Z;t,q) = \sum_\lambda \Ktil_{\lambda \mu}(t,q) s_\lambda(Z).
\]
The (since resolved) \emph{Macdonald positivity conjecture} asserts that $\Ktil_{\lambda \mu}(t,q)$ is a polynomial in $t$ and $q$ with nonnegative integer coefficients. This conjecture suggests that $\Htil_\mu(Z;t,q)$ should have a representation-theoretic interpretation as the character of a doubly graded symmetric group module, and indeed such an interpretation is given by the \emph{Garsia-Haiman module} $\Hcal_\mu$:
\begin{theorem}[\cite{Hai01}]
The transformed Macdonald polynomial $\Htil_\mu(Z;t,q)$ equals the bigraded Frobenius series of the Garsia-Haiman module $\Hcal_\mu$.
\end{theorem}
In fact, Garsia and Haiman were able to show that in order to establish this connection, it would be enough to prove that $\dim(\Hcal_\mu) = n!$. Thus the problem was reduced to this \emph{$n!$ conjecture}, which was later resolved by Haiman (\cite{Hai01}).

This paper studies the related \emph{$\frac{n!}{k}$ conjecture} of Bergeron and Garsia, which makes a similar assertion about the dimension of the intersection of certain $k$-tuples of Garsia-Haiman modules:
\begin{conj}[\cite{BG}]
Let $\lambda$ be an integer partition of $n+1$, and let $\mu^{(1)}, \ldots, \mu^{(k)}$ be partitions of $n$ each obtained from $\lambda$ by removing a removable cell from the Young diagram of $\lambda$. Then,
\[
\dim\left( \bigcap_{i=1}^k \Hcal_{\mu^{(i)}} \right) = \frac{n!}{k}.
\]
\end{conj}
We resolve the $\frac{n!}{k}$ conjecture in the special case that $\lambda$ has hook shape --- the only nontrivial case is $k=2$ since a hook has at most two removable cells, hence the title. The rest of the paper is organized as follows: after outlining some necessary prerequisites in Section \ref{sec:background}, we introduce the Garsia-Haiman module $\Hcal_\mu$ in Section \ref{sec:GH}. In particular, we define a basis for $\Hcal_\mu$ when $\mu$ has hook shape; each basis element is indexed by a \emph{standard filling} of $\mu$, and is given by a monomial which encodes particular inversions in the standard filling. This basis is simply a restatement of the $k$-th Artin basis of \cite{ARR}, using the language of standard fillings rather than permutations. 

Equipped with this basis, we proceed in Section \ref{sec:n-fact} to construct an explicit basis for the intersection of two Garsia-Haiman modules with hook shape by defining a bijective map between particular subsets of standard fillings which leaves the corresponding basis element unchanged. Our bijection makes use of two ``Foata-like" maps which preserve particular inversions in the first row and column of a standard filling of hook shape.

\section{Macdonald polynomials}\label{sec:background}

We begin by cataloguing some necessary terminology. An \emph{integer partition} of $n \in \NN$ is a nonincreasing sequence of positive integers $\mu = (\mu_1 \ge \mu_2 \ge \cdots \ge \mu_k)$ such that $\sum_i \mu_i = n$, abbreviated $\mu \vdash n$. We say that an integer partition $\mu$ has \emph{hook shape} if $\mu = (a, 1, \ldots, 1)$ for some $a \ge 1$. The \emph{Young diagram} of $\mu$ a subset of $\NN \times \NN$ with $\mu_i$ left-justified cells in row $i$. We opt for French (coordinate) notation, so that the partition $\mu = (4,3,3)$ has Young diagram
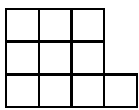
\[
\ytableausetup{aligntableaux=bottom, boxsize=1.2em}
\begin{ytableau}
~ & ~ & ~ \\
~ & ~ & ~ \\
~ & ~ & ~ & ~
\end{ytableau}.
\]
By a slight abuse of notation we identify $\mu$ with its Young diagram; e.g. when we refer to ``the cells of $\mu$" we really mean ``the cells of the Young diagram of $\mu$".

The \emph{length} of a partition $\mu$, denoted $\ell(\mu)$, is the number of nonzero rows of $\mu$, and the \emph{conjugate} of $\mu$, denoted $\mu'$, is the partition obtained from $\mu$ by transposing rows and columns. We say a cell $c$ in $\mu$ is \emph{removable} if there are no cells immediately above or directly right of $c$; in other words, $c$ is removable if the diagram $\mu - \{ c \}$ still has partition shape. For instance, the cells at the ends of rows $1$ and $3$ in the above diagram are removable, but the cell at the end of row $2$ is not.

A \emph{filling} of $\mu \vdash n$ is a map $S : \mu \to \NN$ --- i.e. an assignment of positive integers to the cells of $\mu$ --- and a \emph{standard filling} of $\mu$ is a bijective map $S : \mu \to \{ 1, 2, \ldots, n \}$. For instance, the six standard fillings of $\mu = (2,1)$ are:
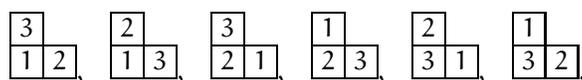
\[
\ytableausetup{aligntableaux=bottom, boxsize=1.2em}
\begin{array}{cccccc}
\begin{ytableau}
3 \\
1 & 2
\end{ytableau},
&
\begin{ytableau}
2 \\
1 & 3
\end{ytableau},
&
\begin{ytableau}
3 \\
2 & 1
\end{ytableau},
&
\begin{ytableau}
1 \\
2 & 3
\end{ytableau},
&
\begin{ytableau}
2 \\
3 & 1
\end{ytableau},
&
\begin{ytableau}
1 \\
3 & 2
\end{ytableau}.
\end{array}
\]
For a filling $S$, we let $S_{i,j}$ denote the entry in row $i$, column $j$ of $S$.

Let $Z = z_1, z_2, \ldots$ denote a countably infinite set of variables. The \emph{symmetric group} $\Sym_n$ is the group of permutations of $[n] = \{ 1, 2, \ldots, n \}$, and we say that a function $f \in \QQ[Z]$ is \emph{symmetric} if it is unchanged under the $\Sym_n$-action given by permuting indices, for any $n \in \NN$. Let $\Lambda \subset \QQ[Z]$ denote the subring of symmetric functions. Many bases for $\Lambda$ are known, but perhaps the most famous (and most useful from a representation-theoretic point of view) is the basis of \emph{Schur functions} $\{ s_\mu \}$, where $\mu$ ranges over the collection of all integer partitions.

Macdonald (\cite{Mac88}) defined an exceptional new basis for the ring $\Lambda_{(t,q)}$ of symmetric functions with coefficients in the field $\QQ(t.q)$ which simultaneously generalizes the Schur functions, Hall-Littlewood polynomials, and Jack symmetric functions, among others. Macdonald's polynomials were not given by an explicit formula --- rather, they were characterized by particular triangularity and orthogonality relations with respect to a $(t,q)$-deformation of the Hall inner product --- but a certain transformed version $\Htil_\mu(Z;t,q)$ of the Macdonald polynomial has an elegant combinatorial description, proved in \cite{HHL}. For any integer partition $\mu$, $\Htil_\mu(Z;t,q)$ may be described as a $(t,q)$-weighted sum over fillings of $\mu$, where the $t$- and $q$-weights are given by combinatorial statistics on fillings which were originally defined by Haglund (\cite{Hag04}). See \cite{Hag04} or \cite{HHL} for the precise definition of these statistics.

Moreover, $\Htil_\mu(Z;t,q)$ has an elegant expansion in the Schur basis:
\begin{equation}\label{eq:schur-exp}
\Htil_\mu(Z;t,q) = \sum_\lambda \Ktil_{\lambda \mu}(t,q) s_\lambda(Z).
\end{equation}
The coefficients $\Ktil_{\lambda \mu}(t,q)$ are \emph{a priori} rational functions in $t$ and $q$, but Macdonald conjectured that $\Ktil_{\lambda \mu}(t,q) \in \NN[t,q]$; this \emph{Macdonald positivity conjecture} was resolved in \cite{Hai01}, by realizing $\Ktil_{\lambda \mu}(t,q)$ as the doubly graded character multiplicities of a certain doubly graded $\Sym_n$-module, which we describe in the following section.

\section{Garsia-Haiman modules}\label{sec:GH}

Garsia and Haiman (\cite{GH93}) proposed the following representation-theoretic interpretation for $\Htil_\mu(Z;t,q)$. Consider the polynomial ring $\QQ[X,Y] = \QQ[x_1, \ldots, x_n ; y_1, \ldots, y_n]$ along with the diagonal $\Sym_n$-action given by permuting the $X$- and $Y$-variables simultaneously and identically; that is,
\[
\sigma \cdot f(x_1, \ldots, x_n ; y_1, \ldots, y_n) = f(x_{\sigma(1)}, \ldots, x_{\sigma(n)} ; y_{\sigma(1)}, \ldots, y_{\sigma(n)})
\] 
for any $f(X,Y) \in \QQ[X,Y]$ and $\sigma \in \Sym_n$. For $\mu \vdash n$, define the polynomial
\[
\Delta_\mu = \det(x_i^{p_j-1}y_i^{q_j-1})_{i,j=1}^n \in \QQ[X,Y],
\]
where $\{ (p_1,q_1), \ldots, (p_n, q_n) \}$ encode the coordinates of the cells of $\mu$, taken in any order. For instance, we label the cells of $\mu = (3,2)$ as follows:
\[
\ytableausetup{aligntableaux=bottom, boxsize=2em}
\begin{ytableau}
\scriptstyle{(2,1)} & \scriptstyle{(2,2)} \\
\scriptstyle{(1,1)} & \scriptstyle{(1,2)} & \scriptstyle{(1,3)}
\end{ytableau}.
\]
Then $\Delta_\mu \neq 0$ since the above biexponents are distinct, and $\Delta_\mu$ is an $\Sym_n$-alternating polynomial which is doubly homogeneous of $X$-degree $b(\mu) = \sum_{i=1}^{\ell(\mu)} (i-1) \mu_i$ and of $Y$-degree $b(\mu') = \sum_{j=1}^{\ell(\mu')} (j-1) \mu'_j$. Let
\[
\ii_\mu = \{ f \in \QQ[X,Y] : f(\tfrac{\partial}{\partial x_1}, \ldots, \tfrac{\partial}{\partial x_n}; \tfrac{\partial}{\partial y_1}, \ldots, \tfrac{\partial}{\partial y_n}) \Delta_\mu = 0 \}
\]
denote the ideal of polynomials whose corresponding differential operator annihilates $\Delta_\mu$. Note that, since $\Delta_\mu$ is doubly homogeneous and $\Sym_n$-alternating, $\ii_\mu$ is an $\Sym_n$-invariant, doubly homogeneous ideal.
\begin{definition}[\cite{GH93}]
The \emph{Garsia-Haiman module} is the quotient ring
\[
\Hcal_\mu = \QQ[X,Y] / \ii_\mu
\]
viewed as a doubly graded $\Sym_n$-module, with the $\Sym_n$-action induced by the diagonal action on $\QQ[X,Y]$.
\end{definition}
Since $\Hcal_\mu$ is doubly graded, we have a direct sum decomposition
\[
\Hcal_\mu = \bigoplus_{i=1}^{b(\mu)} \bigoplus_{j=1}^{b(\mu')} (\Hcal_\mu)_{i,j},
\]
where $(\Hcal_\mu)_{i,j}$ denotes the submodule spanned by elements of total $X$-degree $i$ and total $Y$-degree $j$.

Garsia and Haiman conjectured that the \emph{Frobenius series} of $\Hcal_\mu$ (defined below) is precisely $\Htil_\mu(Z;t,q)$. Due to known identities involving the usual Kostka numbers $K_{\lambda \mu}$, it is necessary that $\Hcal_\mu$ afford a doubly graded version of the regular representation of $\Sym_n$, and in particular that the dimension of $\Hcal_\mu$ is $n!$, for this identity to hold. Strikingly, this apparently weaker condition is sufficient to establish the desired result:
\begin{theorem}[\cite{Hai99}]
If the dimension of $\Hcal_\mu$ is $n!$, then the bigraded Frobenius series of $\Hcal_\mu$ given by
\[
Frob_{\Hcal_\mu}(Z;t,q) = \sum_{i,j} t^iq^j ch((\Hcal_\mu)_{i,j}),
\]
where $ch$ is the map which sends the irreducible $\Sym_n$-representation $S^\lambda$ to the Schur function $s_\lambda$, equals the transformed Macdonald polynomial $\Htil_\mu(Z;t,q)$.
\end{theorem}
Thus the problem was reduced to this comparatively simpler \emph{$n!$ conjecture}, which was later resolved by Haiman:
\begin{theorem}[\cite{Hai01}]
The dimension of $\Hcal_\mu$ is $n!$.
\end{theorem}
This in turn proves the Macdonald positivity conjecture, since it realizes $\Ktil_{\lambda \mu}(t,q)$ as the doubly graded multiplicity of $S^\lambda$ in $\Hcal_\mu$. The proof of the $n!$ conjecture in \cite{Hai01}, however, relies on deep algebro-geometric results, and it remains an open problem to construct an explicit basis for $\Hcal_\mu$ in general.

Many bases for $\Hcal_\mu$ are known when $\mu$ has hook shape, however (see \cite{ARR},\cite{All},\cite{AG},\cite{Aval},\cite{GH96},\cite{Stem}). We opt to work with the $k$-th Artin basis of \cite{ARR}, which was also known to Garsia and Haiman in \cite{GH95}. Using the following terminology, we reframe this basis in a form which is more amenable to our later constructions.


\begin{definition}
Let $SF(\mu)$ denote the set of standard fillings of $\mu$. Given $S \in SF(a,1^\ell)$, define a \emph{row inversion} to be a pair $S_{1,i} > S_{1,j}$ where $i < j$. Similarly, define a \emph{column inversion} to be a pair $S_{j,1} > S_{i,1}$ for $i < j$ (note the swapped indices). Denote the collection of row (resp. column) inversions in $S$ by
\begin{align*}
\mathsf{rowInv}(S) &= \{ (t,r) : t > r, t \text{ left of } r \text{ in row $1$ of } S \}, \\
\mathsf{colInv}(S) &= \{ (d,c) : d > c, d \text{ above } c \text{ in column $1$ of } S \}.
\end{align*}
\end{definition}

\begin{definition}
Given $S \in SF(a,1^\ell)$, define the polynomial $\varphi_S(X,Y) \in \QQ[X,Y]$ by
\[
\varphi_S(X,Y) = \prod_{(d,c) \in \mathsf{colInv}(S)} x_d \prod_{(t,r) \in \mathsf{rowInv}(S)} y_r
\]
(cf. \cite{ARR}, Definition 1.6).
\end{definition}

\begin{example}
The following standard filling of $\mu = (5,1^4)$:
\[
S =
\ytableausetup{centertableaux, boxsize = 1.2em}
\begin{ytableau}
4 \\
1 \\
9 \\
7 \\
5 & 6 & 3 & 2 & 8
\end{ytableau} 
\]
has
\begin{align*}
\mathsf{rowInv}(S) &= \{ (5,\mathbf{3}), (5,\mathbf{2}), (6,\mathbf{3}), (6,\mathbf{2}), (3,\mathbf{2}) \}, \\
\mathsf{colInv}(S) &= \{ (\mathbf{4},1), (\mathbf{9},7), (\mathbf{9},5), (\mathbf{7},5) \},
\end{align*}
so $\varphi_S(X,Y) = x_4x_7x_9^2y_2^3y_3^2$.
\end{example}

\begin{theorem}[\cite{ARR}, Corollary 1.7]
The set
\[
\{ \varphi_S(X,Y) : S \in SF(a,1^\ell) \}
\]
forms a basis for the Garsia-Haiman module $\Hcal_{(a,1^\ell)}$.
\end{theorem}

\begin{remark}
The theory of orbit harmonics developed by Garsia and Haiman (\cite{GH}) lays the groundwork for a new, direct proof that $\{ \varphi_S \}$ forms a basis for $\Hcal_\mu$, which we sketch here. Let $\alpha_1, \ldots, \alpha_n, \beta_1, \ldots, \beta_n$ be distinct rational numbers. To any standard filling $S$ we associate an \emph{orbit point} $p_S = (\alpha_{i_1}, \ldots, \alpha_{i_n}; \beta_{j_1}, \ldots, \beta_{j_n})$, and a polynomial
\[
\psi_S = \prod_{(d,c) \in \mathsf{colInv}(S)} (x_d - \alpha_{\mathsf{row}(c)}) \prod_{(t,r) \in \mathsf{rowInv}(S)} (y_r - \beta_{\mathsf{col}(t)}),
\]
so that $\varphi_S$ is the leading term of $\psi_S$. Garsia and Haiman proved that if
\begin{enumerate}
\item the matrix $(\psi_S(p_T))$ is nonsingular, and
\item the Hilbert series of $\{ \psi_S \}$ is symmetric,
\end{enumerate}
then the leading terms $\{ \varphi_S \}$ form a basis for $\Hcal_\mu$. There exists a total ordering on standard fillings under which $(\psi_S(p_T))$ is upper-triangular with nonzero diagonal entries (hence nonsingular), and the symmetry of the Hilbert series is evidenced by the fact that $\varphi_S(t,q) = t^{n(\mu)}q^{n(\mu')} \varphi_{\widetilde{S}}(t^{-1},q^{-1})$, where $\widetilde{S}$ is obtained from $S$ by replacing the entry $i$ with $n-i+1$.
\end{remark}

\section{A proof of the $\frac{n!}{2}$ conjecture for hook shapes}\label{sec:n-fact}

Now equipped with an explicit basis for $\Hcal_{(a,1^\ell)}$, we can study the remarkable intersection properties of these modules which were conjectured in \cite{BG}. In full generality, the \emph{$\frac{n!}{k}$ conjecture} is:
\begin{conj}[\cite{BG}]\label{conj:n!/k}
Let $\lambda \vdash n+1$, and let $\mu^{(1)}, \ldots, \mu^{(k)}$ be partitions of $n$ each be obtained from $\lambda$ by removing a removable cell from the Young diagram of $\lambda$. Then,
\[
\dim\left( \bigcap_{i=1}^k \Hcal_{\mu^{(i)}} \right) = \frac{n!}{k}.
\]
\end{conj}

We prove the $\frac{n!}{k}$ conjecture when $\lambda$ has hook shape; in particular, the only nontrivial case is $k=2$, since a hook shape only has two removable cells. These ideas are also pursued in \cite{BH}, where the authors obtain a plethystic formula for the Frobenius series of the intersection of two Garsia-Haiman modules with hook shape. For the rest of the paper, assume $\lambda = (a,1^\ell)$ is a partition of $n+1 := a+\ell$ with $a \ge 2$ and $\ell \ge 1$ --- the conjecture is trivial otherwise --- and let $\mu = (a,1^{\ell-1}), \rho = (a-1,1^\ell)$ be partitions of $n$ obtained from $\lambda$ by removing a cell from the end of the first column, and from the end of the first row, respectively.

\begin{theorem}[$\frac{n!}{2}$ for hook shapes]\label{nfact/2}
For partitions $\mu, \rho$ as defined above,
\[
\dim (\Hcal_\mu \cap \Hcal_\rho) = \frac{n!}{2}.
\]
\end{theorem}
The rest of Section \ref{sec:n-fact} is devoted to a proof of Theorem \ref{nfact/2}.

\begin{definition}
As above, let $SF(\mu)$ (resp. $SF(\rho)$) denote the set of standard fillings of $\mu$ (resp. $\rho$). Define a map $\mathsf{bump} : SF(\mu) \to SF(\rho)$, where $\mathsf{bump}(S)$ is the filling of $\rho$ obtained from $S$ by moving each entry in the first column up one row, and then pushing each remaining entry in the first row to the left one column; that is,
\[
\mathsf{bump} \left(
\ytableausetup{centertableaux, boxsize=2.3em}
\begin{ytableau}
v_{\ell} \\
\vdots \\
v_2 \\
v_1 & u_1 & u_2 & \cdots & u_{a-1}
\end{ytableau}
\right)
=
\begin{ytableau}
v_\ell \\
\vdots \\
v_2 \\
v_1 \\
u_1 & u_2 & \cdots & u_{a-1}
\end{ytableau}.
\]
\end{definition}
\noindent Now consider the subsets
\begin{align*}
SF_<(\mu) &= \{ S \in SF(\mu) : S_{1,1} < S_{1,a} \}, \\
SF_<(\rho) &= \{ T \in SF(\rho) : T_{\ell+1,1} < T_{1,1} \},
\end{align*}
where $S_{i,j}$ denotes the entry in row $i$, column $j$ of $S$.

We have $\# SF_<(\mu) = \# SF_<(\rho) = \frac{n!}{2}$, and we claim furthermore that
\[
\{ \varphi_S : S \in SF_<(\mu) \} = \{ \varphi_T : T \in SF_<(\rho) \}.
\] 
To prove this assertion, we will define a bijective map $\theta : SF_<(\mu) \rightarrow SF_<(\rho)$ such that $\varphi_S = \varphi_{\theta(S)}$ for any $S \in SF_<(\mu)$. 

There are two main steps in defining the map $\theta$:
\begin{enumerate}
\item Given a filling $S \in SF_<(\mu)$, consider $\mathsf{bump}(S) \in SF(\rho)$. This new filling is potentially missing some row inversions in $S$ which were created by $u = S_{1,1}$, which now resides in the second row, so we must rearrange what remains of the first row to reintroduce the necessary inversions.
\item After rearranging the first row there is a new entry $v$ in row $1$, column $1$, so we may have introduced some column inversions which were not there before, and we must rearrange the first column to remove these new inversions.
\end{enumerate}

To this end, we define a map $\mathsf{arm}_u$ which reintroduces the row inversions removed in the first step, and a map $\mathsf{leg}_v$ which removes the column inversions introduced in the second. In what follows, let
\[
\Wcal = \{ w_1w_2 \cdots w_m : m, w_i \in \NN \}
\]
denote the set of words of finite length in the alphabet $\NN$.

\begin{definition}
For a fixed $u \in \NN$, consider the set of words
\[
\Acal_u = \{ w_1w_2\cdots w_m : m \in \NN, w_i \in \NN - \{ u \}, w_m > u \}.
\]
Define a map $\mathsf{arm}_u: \Acal_u \to \Wcal$ as follows: given $w = w_1 w_2 \cdots w_m \in \Acal_u$, let $b_1, b_2, \ldots, b_j$ denote the indices for which $w_{b_i} < u$. If there are no such indices, then define $\mathsf{arm}_u(w) = w$. Otherwise,
\begin{enumerate}
\item Draw a vertical bar immediately to the left of $w_{b_i}$ if either $w_{b_i - 1} > u$, or if $i = 1$ and $b_1 = 1$. 

\item Within each newly created block which contains at least one of the $w_{b_i}$'s, move the leftmost entry which is $> u$ to the front of the block, immediately right of the (leftmost) vertical bar. 
\end{enumerate}
Define the resulting word to be $\mathsf{arm}_u(w)$. This map is well-defined on $\Acal_u$, since the last letter in $w \in \Acal_u$ is required to be $> u$; that is, we will always have an entry $>u$ to shuffle within each block.
\end{definition}

\begin{example}
For $w = 49263187 \in \Acal_5$, draw vertical bars to the left of $4$, $2$, and $3$, and within each block shuffle the leftmost number $>5$ to the front:
\[
\begin{tikzpicture}
\draw[thick] (0.3,-0.3) -- (0.3,0.3);
\draw (0.3,-0.2) node (B1) {};
\draw (0.5,0) node[darkred] {$4$};
\draw (1,0) node (C1) {$9$};
\draw[thick] (1.3,-0.3) -- (1.3,0.3);
\draw (1.3,-0.2) node (B2) {};
\draw (1.5,0) node[darkred] {$2$};
\draw (2,0) node (C2) {$6$};
\draw[thick] (2.3,-0.3) -- (2.3,0.3);
\draw (2.3,-0.2) node (B3) {};
\draw (2.5,0) node[darkred] {$3$};
\draw (3,0) node[darkred] {$1$};
\draw (3.5,0) node (C3) {$8$};
\draw (4,0) node {$7$};
\draw [->,darkgreen] (C1.south) to [out=235,in=305] (B1.south east);
\draw [->,darkgreen] (C2.south) to [out=235,in=305] (B2.south east);
\draw [->,darkgreen] (C3.south) to [out=235,in=305] (B3.south east);

\draw (5.25,0) node {$\longmapsto$};

\draw[thick] (6.3,-0.3) -- (6.3,0.3);
\draw (6.5,0) node {$9$};
\draw (7,0) node {$4$};
\draw[thick] (7.3,-0.3) -- (7.3,0.3);
\draw (7.5,0) node {$6$};
\draw (8,0) node {$2$};
\draw[thick] (8.3,-0.3) -- (8.3,0.3);
\draw (8.5,0) node {$8$};
\draw (9,0) node {$3$};
\draw (9.5,0) node {$1$};
\draw (10,0) node {$7$};
\end{tikzpicture}
\]
So, $\mathsf{arm}_5(49263187) = 94628317$.
\end{example}

\begin{definition}
For a fixed $v \in \NN$, consider the set of words
\[
\Lcal_v = \{ w_1w_2\cdots w_m : m \in \NN, w_i \in \NN - \{ v \}, w_1 < v \}.
\]
Define $\mathsf{leg}_v : \Lcal_v \to \Wcal$ as follows: given $w = w_1 w_2 \cdots w_m \in \Lcal_v$, let $c_1, c_2, \ldots, c_k$ denote the indices for which $w_{c_i} > v$. If there are no such indices, then define $\mathsf{leg}_v(w) = w$. Otherwise,
\begin{enumerate}
\item Draw a vertical bar immediately to the right of $w_{c_i}$ if either $w_{c_i + 1} < v$, or if $i = k$ and $c_k = m$. 

\item Within each newly created block which contains at least one of the $w_{c_i}$'s, move the rightmost entry which is $< v$ to the end of the block, immediately left of the (rightmost) vertical bar. 
\end{enumerate}
Define the resulting word to be $\mathsf{leg}_v(w)$. This map is well-defined on $\Lcal_v$, since the first letter in $w \in \Lcal_v$ is required to be $< v$; that is, we will always have an entry $< v$ to shuffle within each block.
\end{definition}

\begin{example}
For $w = 48731926 \in \Lcal_5$, draw vertical bars to the right of $7, 9$, and $6$, and within each block shuffle the rightmost number $< 5$ to the end:
\[
\begin{tikzpicture}
\draw (0.5,0) node (C1) {$4$};
\draw (1,0) node[darkred] {$8$};
\draw (1.5,0) node[darkred] {$7$};
\draw[thick] (1.7,-0.3) -- (1.7,0.3);
\draw (1.7,-0.2) node (B1) {};
\draw (2,0) node {$3$};
\draw (2.5,0) node (C2) {$1$};
\draw (3,0) node[darkred] {$9$};
\draw[thick] (3.2,-0.3) -- (3.2,0.3);
\draw (3.2,-0.2) node (B2) {};
\draw (3.5,0) node (C3) {$2$};
\draw (4,0) node[darkred] {$6$};
\draw[thick] (4.2,-0.3) -- (4.2,0.3);
\draw (4.2,-0.2) node (B3) {};
\draw [->,darkgreen] (C1.south) to [out=305,in=235] (B1.south west);
\draw [->,darkgreen] (C2.south) to [out=305,in=235] (B2.south west);
\draw [->,darkgreen] (C3.south) to [out=305,in=235] (B3.south west);

\draw (5.25,0) node {$\longmapsto$};

\draw (6 + 0.5,0) node {$8$};
\draw (6 + 1,0) node {$7$};
\draw (6 + 1.5,0) node {$4$};
\draw[thick] (6 + 1.7,-0.3) -- (6 + 1.7,0.3);
\draw (6 + 2,0) node {$3$};
\draw (6 + 2.5,0) node {$9$};
\draw (6 + 3,0) node {$1$};
\draw[thick] (6 + 3.2,-0.3) -- (6 + 3.2,0.3);
\draw (6 + 3.5,0) node {$6$};
\draw (6 + 4,0) node {$2$};
\draw[thick] (6 + 4.2,-0.3) -- (6 + 4.2,0.3);
\end{tikzpicture}
\]
So, $\mathsf{leg}_5(48731926) = 87439162$.
\end{example}
Note the similarity of $\mathsf{arm}_u$ and $\mathsf{leg}_v$ to the maps $\gamma_x$ introduced by Foata (\cite{Foa}), which he used to bijectively prove that the major index and inversion number have the same distribution over $\Sym_n$.
\begin{definition}
Using the maps $\mathsf{arm}_u$ and $\mathsf{leg}_v$ defined above, we define $\theta : SF_<(\mu) \to SF_<(\nu)$ as follows. Given $S \in SF_<(\mu)$, let $u = S_{1,1}$, and let $v$ be the leftmost entry in the first row of $S$ which is $> u$. Then define
\[
\theta(S) = \mathsf{leg}_v \circ \mathsf{arm}_u \circ \mathsf{bump}(S),
\]
where $\mathsf{arm}_u$ acts only on the first row of $\mathsf{bump}(S)$ --- that is, we replace the first row of $\mathsf{bump}(S)$ with $\mathsf{arm}_u(S_{1,2} S_{1,3} \cdots S_{1,a})$ --- and $\mathsf{leg}_v$ acts on the first column of $\mathsf{arm}_u(\mathsf{bump}(S))$, strictly above the first row, so that these entries are replaced by $\mathsf{leg}_v(S_{1,1} S_{2,1} \cdots S_{\ell,1})$, entered from bottom to top.
\end{definition}

Since $S_{1,a} > u$ and $u < v$ for any $S \in SF_<(\mu)$, each step in the composition is defined; thus $\theta$ is well-defined. Furthermore we have $\theta(S)_{\ell + 1,1} < v = \theta(S)_{1,1}$ by construction, so that $\theta(S) \in SF_<(\rho)$.

\begin{example}
Let $\mu = (5,1^4)$ and $\rho = (4,1^5)$, and let
\[
S =
\ytableausetup{centertableaux, boxsize = 1.2em}
\begin{ytableau}
9 \\
1 \\
7 \\
4 \\
5 & 6 & 3 & 2 & 8
\end{ytableau} \in SF_<(\mu).
\]
Then in the language of the above definition, we have $u = 5, v = 6$, so that
\begin{align*}
\theta(S) &= \mathsf{leg}_6 \circ \mathsf{arm}_5 \circ \mathsf{bump}(S) = \mathsf{leg}_6 \circ \mathsf{arm}_5\left(
\begin{ytableau}
9 \\
1 \\
7 \\
4 \\
5  \\
6 & {\color{darkred} 3} & {\color{darkred} 2} & {\color{darkgreen} 8}
\end{ytableau}
\right) \\
&= \mathsf{leg}_6\left(
\begin{ytableau}
{\color{darkred} 9} \\
{\color{darkgreen} 1} \\
{\color{darkred} 7} \\
{\color{darkgreen} 4} \\
5  \\
6 & {\color{darkgreen} 8} & 3 & 2
\end{ytableau}
\right) = 
\begin{ytableau}
{\color{darkgreen} 1} \\
9 \\
{\color{darkgreen} 4} \\
7 \\
5  \\
6 & 8 & 3 & 2
\end{ytableau} \in SF_<(\rho).
\end{align*}
Note that $\varphi_S = \varphi_{\theta(S)} = x_7^2x_9^4y_2^3y_3^2$.
\end{example}

\begin{theorem}\label{l-terms}
For partitions $\mu = (a,1^{\ell - 1})$ and $\rho = (a-1, 1^\ell)$, the map $\theta : SF_<(\mu) \to SF_<(\rho)$ defined above satisfies $\varphi_S = \varphi_{\theta(S)}$ for every $S \in SF_<(\mu)$.
\end{theorem}

\begin{proof}
We begin by comparing the first rows of $S$ and $\theta(S)$ to demonstrate that
\begin{equation}\label{eq:row-invs}
\{ r : (t,r) \in \mathsf{rowInv}(S) \text{ for some } t \} = \{ r' : (t',r') \in \mathsf{rowInv}(\theta(S)) \text{ for some } t' \}.
\end{equation}
Since the smaller entries of the row inversions in $S$ determine the $y$-terms in $\varphi_S$ (resp. $\theta(S)$), it will follow that $\varphi_S$ and $\varphi_{\theta(S)}$ have identical $y$-terms.

First, each row inversion in $\mathsf{bump}(S)$ is also a row inversion in $S$, but $S$ may have additional row inversions created by $u = S_{1,1}$ which are absent in $\mathsf{bump}(S)$. Let $r_1, r_2, \ldots, r_p$ denote the elements in the first row of $S$ which are $< u$, so that
\[
\mathsf{rowInv}(S) = \mathsf{rowInv}(\mathsf{bump}(S)) \cup \{ (u, r_1), (u, r_2), \ldots, (u, r_p) \}.
\]
We claim that performing $\mathsf{arm}_u$ on $\mathsf{bump}(S)$ introduces exactly one new row inversion involving each of $r_1, r_2, \ldots, r_p$, and does not create any other new row inversions, nor alter any existing ones.

Indeed, in performing the ``blocking" procedure of $\mathsf{arm}_u$, each of the $r_i$'s will be contained in a unique block; in each of these blocks we move some number $z > u > r_i$ to the left of each $r_i$ contained in the block, so that exactly one new row inversion is created with each of $r_1, r_2, \ldots, r_p$ in $\mathsf{arm}_u(\mathsf{bump}((S)))$. Furthermore, since $z$ is chosen to be the leftmost entry in each block which is $> u$, 
\begin{enumerate}
\item[(i)] the only new row inversions created within a block are those created with some $r_i$, and
\item[(ii)] no existing row inversions are removed, since $z$ does not move to the left of a number larger than itself.
\end{enumerate}
We only shuffle numbers within blocks which contain at least one of the $r_i$'s, and performing $\mathsf{leg}_v$ on the resulting filling does not affect the first row, so it follows that \eqref{eq:row-invs} holds.

Now we compare the first columns of $S$ and $\theta(S)$ to demonstrate that
\begin{equation}\label{eq:col-invs}
\{ d : (d,c) \in \mathsf{colInv}(S) \text{ for some } c \} = \{ d' : (d',c') \in \mathsf{colInv}(\theta(S)) \text{ for some } c' \}.
\end{equation}
Since the larger entries of the column inversions in $S$ determine the $x$-terms in $\varphi_S$ (resp. $\theta(S)$), it will follow that $\varphi_S$ and $\varphi_{\theta(S)}$ have identical $x$-terms.

First, note that $v = \mathsf{arm}_u(\mathsf{bump}(S))_{1,1}$ by construction. Then every column inversion in $S$ is also a column inversion in $\mathsf{arm}_u(\mathsf{bump}(S))$, except that $\mathsf{arm}_u(\mathsf{bump}(S))$ may have additional column inversions created by $v$ which are not present in $S$. Let $d_1, d_2, \ldots, d_q$ denote the elements in the first column of $\mathsf{arm}_u(\mathsf{bump}(S))$ which are $> v$, so that
\[
\mathsf{colInv}(S) = \mathsf{colInv}(\mathsf{arm}_u(\mathsf{bump}(S))) - \{ (d_1, v), (d_2, v), \ldots, (d_q, v) \}.
\]
We claim that performing $\mathsf{leg}_v$ on $\mathsf{arm}_u(\mathsf{bump}(S))$ removes exactly one column inversion involving each of $d_1, d_2, \ldots, d_q$. When we perform the ``blocking" procedure required by $\mathsf{leg}_v$ in the first column of $\mathsf{arm}_u(\mathsf{bump}(S))$, each of the $d_j$'s is contained in a unique block, and within each block some number $t < v < d_j$ is shuffled above each $d_j$ in the block, so that exactly one column inversion involving each of $d_1, d_2, \ldots, d_q$ is removed. Furthermore, since $t$ is chosen to be the the topmost number $< v$ in its block,
\begin{enumerate}
\item[(i)] the only column inversions removed within a block are those involving some $d_j$, and
\item[(ii)] no new column inversions are created, since $t$ does not move above any number smaller than itself.
\end{enumerate}
Thus \eqref{eq:col-invs} holds, and we conclude that $\varphi_S = \varphi_{\theta(S)}$.
\end{proof}

\begin{cor}
The map $\theta : SF_<(\mu) \to SF_<(\rho)$ is a bijection.
\end{cor}

\begin{proof}
First, we claim that $\theta$ is injective. Indeed, if $\theta(S) = \theta(T)$ for $S, T \in SF_<(\mu)$, then $\varphi_{\theta(S)} = \varphi_{\theta(T)}$. But then by Theorem \ref{l-terms}, we obtain $\varphi_S = \varphi_T$, i.e. $S = T$. Thus $\theta$ is injective, and since $\# SF_<(\mu) = \# SF_<(\rho) = \frac{n!}{2}$, the result follows.
\end{proof}

\begin{cor}
We have $\dim (\Hcal_\mu \cap \Hcal_\rho) \ge \frac{n!}{2}$.
\end{cor}
\begin{proof}
We have exhibited that $\varphi_S \in \Hcal_\mu \cap \Hcal_\rho$ for any $S \in SF_<(\mu)$ (equivalently $SF_<(\rho)$), and $\{ \varphi_S : S \in SF_<(\mu) \}$ is linearly independent, so it follows that
\[
\dim (\Hcal_\mu \cap \Hcal_\rho) \ge \# \{ \varphi_S : S \in SF_<(\mu) \} = \frac{n!}{2}.
\]
\end{proof}

In fact, we claim that $\{ \varphi_S : S \in SF_<(\mu) \}$ (or equivalently $\{ \varphi_T : T \in SF_<(\rho) \}$) forms a basis for $\Hcal_\mu \cap \Hcal_\rho$, aided by the following lemma.

\begin{lem}\label{lem-1}
For partitions $\mu = (a,1^{\ell-1}), \rho = (a-1,1^\ell)$, let $SF_>(\mu) = SF(\mu) - SF_<(\mu)$ and $SF_>(\rho) = SF(\rho) - SF_<(\rho)$. Then
\begin{enumerate}
\item for any $S \in SF_>(\mu)$, we have $\varphi_S \neq \varphi_T$ for all $T \in SF(\rho)$, and
\item for any $T \in SF_>(\rho)$, we have $\varphi_T \neq \varphi_S$ for all $S \in SF(\mu)$.
\end{enumerate}
\end{lem}

\begin{proof}
For the first assertion, let $S \in SF_>(\mu)$. We then have that $S_{1,1} > S_{1,a}$. Let $r = S_{1,a}$, and let $t_1, t_2, \ldots, t_k$ denote the entries in row $1$ of $S$ which are $< S_{1,1}$, ordered in such a way that $t_1, t_2, \ldots, t_j$ are the entries which are also $> r$. Then $S_{1,1}$ forms a row inversion with each of $t_1, t_2, \ldots, t_k$, and $t_1, t_2, \ldots, t_j$ each form a row inversion with $r$. Since there are then $a-k-1$ remaining terms in row $1$ which are $\ge S_{1,1}$ (and thus form row inversions with $r$), the following term divides $\varphi_{S}$:
\begin{equation}\label{eq:bad-term}
m_S(Y) = y_r^{a-k-1+j} y_{t_1} y_{t_2} \cdots y_{t_k}.
\end{equation}
We now claim that $m_S(Y)$ does not divide $\varphi_T$ for any filling $T \in SF(\rho)$. If it did, then each of $r, t_1, t_2, \ldots, t_k$ must appear in the first row of $T$. Furthermore, since $t_1, t_2, \ldots, t_j$ are the only of these elements which could be the larger entry in a row inversion involving $r$, there must be at least $a-k-1$ other entries in row $1$ which are $> r$ in order for $\varphi_T$ to have a $y_r$-degree of $a-k-1+j$. There are thus at least $a$ distinct entries in the first row of $T$, a contradiction since $\rho = (a-1,1^\ell)$. Thus $\varphi_S \neq \varphi_T$ for all $T \in SF(\rho)$.

The second assertion follows from a completely analogous argument, where we instead show that for any $T \in SF_>(\rho)$, one can infer from $\varphi_T$ that there must be $\ell + 1$ distinct entries in the first column of $T$.
\end{proof}

\begin{theorem}\label{at-most}
For $\mu = (a, 1^{\ell - 1}), \rho = (a-1,1^\ell)$, we have $\dim(\Hcal_\mu \cap \Hcal_\rho) \le \frac{n!}{2}$.
\end{theorem}

\begin{proof}
We demonstrate that no nontrivial linear combination of $\{ \varphi_S : S \in SF_>(\mu) \}$, or of $\{ \varphi_T : T \in SF_>(\rho) \}$, lies in $\Hcal_\mu \cap \Hcal_\rho$. First, assume that
\[
\sum_{S \in SF_>(\mu)} c_S \varphi_S \in \Hcal_\mu \cap \Hcal_\rho
\]
for some $c_S \in \QQ$. Then, there exist constants $d_T \in \QQ$ so that
\[
\sum_{S \in SF_>(\mu)} c_S \varphi_S = \sum_{T \in SF(\rho)} d_T \varphi_T.
\]
Pick any $\widetilde{S} \in SF_>(\mu)$ and write
\[
c_{\widetilde{S}} \varphi_{\widetilde{S}} = \sum_{T \in SF(\rho)} d_T \varphi_T - \sum_{S \in SF_>(\mu) - \{ \widetilde{S} \}} c_S \varphi_S.
\]
By Lemma \ref{lem-1}, and using the fact that $\{ \varphi_S : S \in SF(\mu) \}$ and $\{ \varphi_T : T \in SF(\rho) \}$ are linearly independent sets of monomials, the monomial $m_{\widetilde{S}}(Y)$ defined in \eqref{eq:bad-term} must divide each term on the right-hand side which has a nonzero coefficient. (More precisely, none of the terms \emph{within} a given sum can cancel because they are linearly independent, and none of the terms \emph{between} the sums can cancel because $\varphi_T$ cannot contain the same monomial $m_S(Y)$ as $\varphi_S$ for any $S \in SF_>(\mu)$.)

But, again applying Lemma \ref{lem-1}, $m_{\widetilde{S}}(Y)$ cannot divide $\varphi_T$ for any $T \in SF(\rho)$, forcing $d_T = 0$ for all $T$. Thus the above equation reduces to
\[
c_{\widetilde{S}} \varphi_{\widetilde{S}} = - \sum_{S \in SF_>(\mu) - \{ \widetilde{S} \}} c_S \varphi_S,
\]
forcing $c_S = 0$ for all $S$, since $\{ \varphi_S : S \in SF_>(\mu) \}$ is linearly independent. Thus
\[
\sum_{S \in SF_>(\mu)} c_S \varphi_S \in \Hcal_\mu \cap \Hcal_\rho \Rightarrow c_S = 0 \text{ for all } S.
\]
An identical argument shows that
\[
\sum_{T \in SF_>(\rho)} d_T \varphi_T \in \Hcal_\mu \cap \Hcal_\rho \Rightarrow d_T = 0 \text{ for all } T
\]
by appealing to the second assertion in Lemma \ref{lem-1}. Thus no nonzero element of $\Hcal_\mu \cap \Hcal_\rho$ can be written as a linear combination of $\{ \varphi_S : S \in SF_>(\mu) \}$ or of $\{ \varphi_T : T \in SF_>(\rho) \}$, proving that $\dim(\Hcal_\mu \cap \Hcal_\rho) \le \frac{n!}{2}$.
\end{proof}

Combining Theorems \ref{l-terms} and \ref{at-most} proves Theorem \ref{nfact/2}.

\begin{remark}\label{rem:other-one}
Dually, one may define the Garsia-Haiman module as
\[
D_\mu = \text{span}\{  f(\tfrac{\partial}{\partial x_1}, \ldots, \tfrac{\partial}{\partial x_n}; \tfrac{\partial}{\partial y_1}, \ldots, \tfrac{\partial}{\partial y_n}) \Delta_\mu : f \in \QQ[X,Y] \},
\]
the span of all partial derivatives of all orders of $\Delta_\mu$; the map $\Hcal_\mu \to D_\mu$ given by $f(X,Y) \mapsto f(\frac{\partial}{\partial X}, \frac{\partial}{\partial Y})\Delta_\mu$ is an isomorphism of doubly graded $\Sym_n$-modules (see \cite{Hai99}, Proposition 3.4).

It is worth pointing out that the original $\frac{n!}{k}$ conjecture in \cite{BG} is stated in terms of $D_\mu$, rather than $\Hcal_\mu$. Our basis for $\Hcal_\mu \cap \Hcal_\rho$ does not immediately yield a basis for $D_\mu \cap D_\rho$ in the sense that
\[
\{ \varphi_S(\tfrac{\partial}{\partial X},\tfrac{\partial}{\partial Y}) \Delta_\mu : S \in SF_<(\mu) \} \neq \{ \varphi_T(\tfrac{\partial}{\partial X},\tfrac{\partial}{\partial Y}) \Delta_\rho : T \in SF_<(\rho) \}.
\]
In fact, experimental evidence suggests that these sets are disjoint.

As the above isomorphism is bidegree-complementing, the naive/optimistic conjecture would be that the complements $SF_>(\mu)$ and $SF_>(\rho)$ index a basis for $D_\mu \cap D_\rho$. However, it is not true in general that $\varphi_S(\frac{\partial}{\partial X},\frac{\partial}{\partial Y})\Delta_\mu \in D_\rho$ for arbitrary $S \in SF_>(\mu)$ (and vice-versa), either. Thus it appears to be nontrivial to construct an explicit basis for $D_\mu \cap D_\rho$ in the present paradigm.
\end{remark}

\bibliographystyle{plain}

\bibliography{macdonald}

\end{document}